\documentclass{amsart}

\newcommand{\beqa}{\begin{eqnarray*}}
\newcommand{\eeqa}{\end{eqnarray*}}
\newcommand{\beqn}{\begin{eqnarray}}
\newcommand{\eeqn}{\end{eqnarray}}

\newcommand{\ra}{\rightarrow}

\newcommand{\N}{\mathbb N}

\newcommand{\bT}{\mathbb T}

\newcommand{\f}{\frac}

\newcommand{\ph}{\phi}

\newcommand{\la}{\lambda}
\newcommand{\La}{\Lambda}

\newcommand{\Om}{\Omega}
\newcommand{\om}{\omega}

\newcounter{cnt1}
\newcounter{cnt2}
\newcounter{cnt3}
\newcommand{\blr}{\begin{list}{$($\roman{cnt1}$)$}
 {\usecounter{cnt1} \setlength{\topsep}{0pt}
 \setlength{\itemsep}{0pt}}}
\newcommand{\bla}{\begin{list}{$($\alph{cnt2}$)$}
 {\usecounter{cnt2} \setlength{\topsep}{0pt}
 \setlength{\itemsep}{0pt}}}
\newcommand{\bln}{\begin{list}{$($\arabic{cnt3}$)$}
 {\usecounter{cnt3} \setlength{\topsep}{0pt}
 \setlength{\itemsep}{0pt}}}
\newcommand{\el}{\end{list}}

\newtheorem{thm}{Theorem}[section]
\newtheorem{lem}[thm]{Lemma}
\newtheorem{cor}[thm]{Corollary}
\newtheorem{ex}[thm]{Example}

\newtheorem{Def}[thm]{Definition}
\newtheorem{prop}[thm]{Proposition}
\newtheorem{rem}[thm]{Remark}

\newcommand{\Rem}{\begin{rem} \rm}
\newcommand{\bdfn}{\begin{Def} \rm}
\newcommand{\edfn}{\end{Def}}

\newcommand{\ba}{\begin{array}}
\newcommand{\ea}{\end{array}}

\newtheorem*{ack}{Acknowledgements}

\begin{document}
\sloppy

\title[]{Representation of Generalized Bi-Circular Projections on Banach Spaces}

\author{A. B. Abubaker*}
\address[Abdullah Bin Abu Baker]{Department of Mathematics and Statistics
Indian Institute of Technology Kanpur
Kanpur - 208016
India}
\email{abdullah@iitk.ac.in }

\author{Fernanda Botelho}
\address[Fernanda Botelho]{Department of Mathematical Sciences, The University of Memphis, Memphis, TN 38152, USA}
\email{mbotelho@memphis.edu}

\author{James Jamison}
\address[James Jamison]{Department of Mathematical Sciences, The University of Memphis, Memphis, TN 38152, USA}
\email{jjamison@memphis.edu}

\thanks{* Supported by the Indian Institute of Technology Kanpur, Kanpur,
India and by the National Board of Higher Mathematics,
Department of Atomic Energy, Government of India, grant No.
2/44(56)/2011-R\&D-II/3414}



\date{\today}

\begin{abstract}

We prove several results concerning the representation of projections on arbitrary Banach spaces.
We also give illustrative examples including an example of a generalized bi-circular projection
which can not be written as the average of the identity with an isometric reflection. We also
characterize generalized bi-circular projections on $C_0(\Om,X)$, with $\Om$  a locally compact
Hausdorff space (not necessarily connected) and $X$ a Banach space with trivial centralizer.

\end{abstract}

\maketitle

\section{Introduction}

A projection $P$ on a complex Banach space $X$ is said to be a bi-circular
projection if $\displaystyle e^{ia}P + e^{ib}(I - P)$ is an
isometry, for all choices of real numbers $a$ and $b$. These
projections were first studied by Stacho and Zalar (in \cite{SZ1}
and \cite{SZ2}) and shown to be norm hermitian by Jamison (in \cite{J}).

Fo\v sner, Ili\v sevic and Li introduced a larger class of projections
designated generalized bi-circular projections (henceforth GBP),
cf. \cite{FIL}. A generalized bi-circular projection $P$ only requires
that $P + \lambda (I - P)$ is an isometry, for some $\la \in \bT \setminus \{1\}$.
These projections are not necessarily norm hermitian.
It is a consequence of the definition of a GBP that $P + \lambda (I-P)$ must be a
surjective isometry, since
\[ (P + \lambda (I-P))(y)= \, x,\ \mbox{where} \ y = Px + \frac{1}{ \lambda}
(I-P)x , \,\,\, \forall \,\, x \in X.\]

In \cite{FIL}, the authors show that a generalized bi-circular projection
on finite dimensional spaces is equal to  the average of the identity with
an isometric reflection. This interesting result was extended further to
many other settings, as for example  spaces of continuous functions on a
compact, connected and  Haursdorff space, $C(\Omega)$ and $C(\Omega,X)$,
where generalized bi-circular projections are also represented as the average
of the identity with an isometric reflection, see \cite{BJ1} and  \cite{FR}.
The same characterization also holds for generalized bi-circular projections
on spaces of Lipschitz functions, see \cite{BJ2} and \cite{VCV} and for
$L_p$-spaces, $1 \leq p < \infty, p \neq 2$, see \cite{L}.
This raises the question whether every GBP on a Banach space is equal to the
average of the identity with an isometric reflection. The answer to this question
is negative as we show in example (\ref{gbp_ne_(I+R):2}).

It is easy to see that  there is a bijection between the set of all reflections on $X$
and the set of all projections on $X$. If  $P=\frac{Id+R}{2}$,
with $R$ an isometric reflection, is a GBP, then $R$ is   the identity
on the range of $P$ and  $-I$ on the kernel of $P$.

In this note we show that given a GBP $P$ on an arbitrary
complex Banach space, $P$ is hermitian or $P$  is the average of the identity with a
reflection $R,$ with $R$ an element in the algebra generated by the isometry associated
with $P$. We give examples that show that the reflection defined by a GBP is not necessarily
an isometry. Moreover, we also show that every projection on $X$ is a GBP relative to some
renorming of the underlying space $X$. Therefore in this new  space, $P$ can be
represented as the average of the identity with an isometric  reflection.

In section 3 we characterize projections written as combinations of
iterates of a finite order operator and we relate those to the
generalized $n$-circular projections discussed in \cite{B} and also
in \cite{AD}. In section 4 we derive the standard form for
generalized bi-circular projections on $C_0(\Om,X)$, with $\Om$  a
locally compact Hausdorff space (not necessarily connected) and $X$
a Banach space with trivial centralizer.

\section{A characterization of generalized bi-circular
projection on a complex Banach space}

Throughout this section $X$ denotes a complex Banach space and $P$ a
bounded linear projection on $X$. We recall that $P$ is a generalized
bi-circular projection if and only if there exists a modulus 1
complex number $\lambda \neq 1$ such that $P + \lambda (I-P)$ is an
isometry on $X$.

We observe that given an arbitrary  projection $P$ on $X$, $2P-I$ is a reflection and
thus $P$ can be represented as the average of the $I$ with a reflection, i.e.
$P= \frac{I+ (2P-I)}{2}.$
In particular,  generalized bi-circular
projections on $X$ are  averages of the identity with  reflections.
We  recall that a reflection $R$ on $X$ is  a bounded linear
operator such that $R^2=I.$ An isometric reflection is both a reflection
and an isometry. The next result represents the reflection determined by a GBP in
terms of the surjective isometry defined by the projection.

\begin{prop} \label{mt1}

Let $X$ be a Banach space. If $P$ is a projection such that
$P+\la (I-P)=T$, where $\la \in  \bT \setminus \{1\}$ and $T$ is
an isometry on $X,$ then
$P=\f{I+R}{2}$, with  $R,$  a reflection on $X,$ in the algebra generated by $T$.

\end{prop}

\begin{proof}

Since $\la$ is a modulus one complex number, it is of the form
$e^{2\pi \theta i}$ with $\theta$ a real number in the interval $[0,1)$.
Therefore, we consider the following  two cases: (i) $\theta$
is an irrational number, and (ii) $\theta$ is a rational number.
If $\theta$ is an irrational, then the sequence $\{\la ^{n}\}_n$ is
dense in the unit circle. This implies that  $P$ is  a bi-circular
projection since for every $\alpha \in \bT $,  $P+\alpha (I-P)$ is a
surjective isometry, cf. \cite{L}. If $\theta$ is a rational number,
we first assume that $\la$ is of even order. Thus for some
positive integer $k$, $\la^{k}=-1$,
$P+\la^k (I-P)=T^k$ and $P+\la^{2k} (I-P)=I=T^{2k}$. Consequently, $P$
is represented as the average of the identity with the isometric reflection
$T^k$. If $\la$ is of odd order, let $2k+1$ be the smallest positive
integer such that $\la^{2k+1}=1$. Therefore
\begin{equation} \label{eq1} P+\la^j (I-P)=T^j, \,\, \forall \,\,j=1,
\ldots 2k+1. \end{equation}
This implies that $T^{2k+1}=I.$ Furthermore adding the equations
displayed in (\ref{eq1}), we get
\[ (2k+1) P+ (1+\lambda+ \lambda^2+\cdots +\lambda^{2k})(I-P)=
I+T+T^2 + \ldots + T^{2k}.\]
This equation becomes
\begin{equation} \label{eq2} (2k+1) P=I+T+T^2 + \ldots + T^{2k}, \end{equation}
since $1+\lambda+ \lambda^2+\cdots +\lambda^{2k}=0.$
The equation displayed in (\ref{eq2}) implies that
\[ P= \frac{1}{2k+1} \left( I + T + \cdots + T^{2k} \right) = \frac{I + R } {2}, \]
with $ \begin{displaystyle} R = \frac{(1-2k) I + 2T + \cdots + 2T^{2k}}{2k+1}.\end{displaystyle}$

It is a straightforward calculation to check that $R^2=I.$
This completes the proof.

\end{proof}

\begin{rem}

It follows from the proof given for the Theorem \ref{mt1} that for $\theta$
 irrational the projection $P$ is bi-circular then hermitian.
We now give an example that shows the converse of the implication in Theorem
\ref{mt1} does not hold. We consider $X$ the space of all convergent sequences
in $\mathbb{C}$ with the sup norm. Let $T: X \rightarrow X$ be given by
$T (x_1, x_2, x_3, x_4 , \cdots )= (x_2, x_3, x_1, x_4, \, \cdots),$ which involves a
permutation of the first three positions of a sequence in $X$ and the identity at any other position.  It is
clear that $T$ is a surjective isometry and $P = \f{I+T+T^2}{3}$ is a
projection. As defined in the proof given for the Theorem \ref{mt1}, we set
$R= \f{-I+2T+2T^2}{3}.$  The projection $P$ is equal to $\f{I+R}{2}$ and
$R:X \rightarrow X$ is s.t. $R(x_1, x_2, x_3, x_4 , \cdots )= \f{1}{3}(-x_1+2x_2+2x_3,
\, 2x_1-x_2+2x_3,\,  2x_1+2x_2-x_3, 3x_4, \, \cdots).$ Therefore, $R(0,1,1,0, \cdots )=
\f{1}{3} (4,1,1,0,0 \cdots )$.
This shows that $R$ is not an isometry. It is easy to check that $P$ is not a GBP.
Given $\lambda$ of modulus $1$ and $\lambda\neq 1$, we set
$S=(1-\lambda) P + \lambda \, I.$ In particular, \[S(1,0,0,0,\cdots )= \left(\frac{1}{3} + \frac{2}{3}
\lambda, \frac{1}{3} - \frac{1}{3} \lambda, \frac{1}{3} - \frac{1}{3} \lambda, 0, \cdots\right).\]
If $S$ was  an isometry on $X$, then $ \max \{ |\frac{1}{3} + \frac{2}{3} \lambda |,|\frac{1}{3} - \frac{1}{3} \lambda| \}=1$.
We observe that $|\frac{1}{3} - \frac{1}{3} \lambda| < 1$ and if $|\frac{1}{3} + \frac{2}{3} \lambda | = 1$,
then $\lambda =1.$   This contradiction shows that $P$ is not a GBP.

\end{rem}

Given a projection it is of interest  to determine whether $P$ is a generalized
bi-circular projection or equivalently whether the reflection determined by $P$ is an isometry.
We address this question in our next result.

\begin{prop}\label{Tisaniso}

Let $X$ be a Banach space. If $P$ is a projection on $X$ such that $T=P+\lambda (I-P)$, for some
$\lambda \in \mathbb{T}\setminus \{1\}$. Then, $T$ is an isometry if and only if
$\|x-y\| = \|x-\lambda y\|$, for every $x \in \, \mbox{Range} (P)$ and
$y \in \, \mbox{Ker} (P)$.

\end{prop}

\begin{proof}

The projection $P$ determines two closed subspaces $\mbox{Range} (P)$ and $\mbox{Ker} (P)$ such that
$X=\mbox{Range} (P) \oplus \mbox{Ker} (P)$. Since $T$ is an isometry, $\|x-y\|=\|Tx-Ty\|$ for every
$x$ and $y$ in $X$. In particular for $x$ in the range of $P$ and $y$ in the kernel of $P$, we have
$Tx= x$ and $Ty=\lambda y$. The converse follows from straightforward computations.

\end{proof}

\begin{rem}\label{gbp_norm_iden}

If is a consequence of Proposition \ref{Tisaniso} that if $P$ is a generalized bi-circular
projection on $X$, then $P$ is the average of the identity with an isometric reflection
if and only if  for every $x \in \, \mbox{Range} (P)$ and
$y \in \, \mbox{Ker} (P)$, $\|x-y\| = \|x+y\|.$

\end{rem}

The next proposition asserts that every projection on a Banach space is a generalized bi-circular projection
in some equivalent renorming of the given space.

\begin{prop} \label{cor_mt1}

Let $X$ be a complex Banach space and $P$ be a projection on $X$. Then $X$ can be equivalently renormed
such that $R$ is an isometric reflection and consequently $P$ is a generalized bi-circular projection.

\end{prop}

\begin{proof}
We set $R=2P-I$.
We  observe that $R^2=I$ which implies that $R$ is bounded and bijective. Then, the Open Mapping Theorem implies that
$R$ is an isomorphism. Therefore, there exist $\alpha$ and $\beta$ positive numbers such that,
for every $x \in X,$
\[ \alpha \|x\| \leq \|R(x)\| \leq \beta \|x\|.\]
We define  $\|x\|_1= \|x\|+\|R(x)\|, $   for all $x \in X.$ This new norm is equivalent to
the original norm on $X$ and $R$ relative to this norm is an isometry. In fact, given
$x \in X$,  $\| R(x)\|_1= \|R(x)\| + \| R(R(x))\| = \|x\|_1.$

\end{proof}

\begin{ex} \label{gbp_ne_(I+R):2}

We now give an example of a GBP that can not be represented as the average of identity
with an isometric reflection.
Let $X$ be $\mathbb{C}^3$ with the max norm, $\|(x,y,z)\|_{\infty} =
\max \{ |x|, |y|, |z|\}$ and $\lambda = exp(\frac{2\pi i}{3})= -\frac{1}{2}+i
\frac{\sqrt{3}}{2}.$ We consider $P$ the following projection on $\mathbb{C}^3:$
\[ P(x,y,z)= \frac{1}{3} ( x+y+z, x+y+z,x+y+z ) .\]

Let $T= P+\lambda (I-P)$ . Straightforward computations imply that
\[T(x,y,z)= \left( a\, x + b\, (y+z),  a\, y + b\, (x+z),a\, z + b\, (x+y) \right),\]
with $a=\frac{i\sqrt{3}}{3}$ and $b=\frac{1}{2} - \frac{\sqrt{3}i}{6}.$

Since  $T(0,0,1)= \left( b, b, a \right)$, $T$ is not an isometry. In fact,
$\|(0,0,1)\|_{\infty} = 1 $ and $\|T(0,0,1)\|_{\infty}= \frac{\sqrt{3}}{3}\neq
\|(0,0,1)\|_{\infty}.$ The isomorphism $T$ has order $3$ since $\lambda^3=1$.

We now renorm $\mathbb{C}^3$ so $T$ becomes an isometry. The new norm is defined
as follows:
\[ \| (x,y,z)\|_* = \max \{ \| (x,y,z)\|_{\infty}, \, \| T(x,y,z)\|_{\infty},
\| T^2(x,y,z)\|_{\infty} \}.\]

Therefore $P$ is a generalized bi-circular projection in $\mathbb{C}^3$ with the
norm $\| \cdot \|_*,$ for $\lambda=exp(\frac{2\pi i}{3}).$   This projection can
not be written as the average of the identity with an isometric reflection. We
assume otherwise, then  $P=\frac{I+R}{2}$ and
$R= \frac{-I+2T+2T^2}{3}$. We now show that $R$ is not an isometry. Previous
calculations imply that  $T(0,0,1)= ( b, b, a )$ and $T^2(0,0,1)= (b^2+2ab,
b^2+2ab, a^2+2b^2)= (\overline{b},\overline{b},\overline{a}).$ Therefore
$R(0,0,1)=(2/3,2/3,-1/3)$, $(TR)(0,0,1) = \frac{1}{3} \left( \frac{1}{2}+
\frac{\sqrt{3}}{2} i, \frac{1}{2}+ \frac{\sqrt{3}}{2} i, 2-i\sqrt{3}
\right)$, and $(T^2 R)(0,0,1) = \frac{1}{3} \left( \frac{1}{2}-
\frac{\sqrt{3}}{2} i, \frac{1}{2}- \frac{\sqrt{3}}{2} i, 2+i\sqrt{3} \right)$.

Since $\|R(0,0,1)\|_{\infty} = \frac{2}{3}$, $\|(TR)(0,0,1)\|_{\infty} =\|
(T^2 R)(0,0,1) \|_{\infty}= \frac{\sqrt{7}}{3},$ we now conclude that
\[ \|R(0,0,1)\|_* = \max \left\{\frac{2}{3}, \frac{\sqrt{7}}{3}\right\}= \frac{\sqrt{7}}{3}
\neq \|(0,0,1)\|_*=1.\]

\end{ex}

It is worth mentioning that  the projection $P$ above does not satisfy the condition
stated in Remark \ref{gbp_norm_iden}.  For example, if $x=(1,1,1) \in Range(P)$,
$y=(1,1,-2) \in Ker(P),$ we have $\|x+y\|_* =\sqrt{7}$ and $\|x-y\|_*=3.$

\section{Projections as combinations of finite order operators}

In this section we investigate the existence of projections defined as linear
combinations of the iterates of a given finite order operator. We conclude in
our forthcoming Proposition \ref{mp1} that only  certain  averages yield
projections. For a generalized bi-circular projection $P,$ we consider the set
$\Lambda_P = \{ \la \in  \bT : P+\la (I-P) \,\, \mbox{is an isometry}\}$.
This set is a group under multiplication. An inspection of the proof provided
for the Theorem \ref{mt1} also shows that the multiplicative group associated
with a GBP is either finite or equal to $\bT$.  If $\Lambda_P$ is infinite,
then $P$ is  a bi-circular projection. We give some examples of GBPs together
with their multiplicative groups.

\begin{ex}

\begin{enumerate}

\item  We consider $\ell_{\infty}$ with the usual sup norm. Let $P$ be
defined as follows:
\[ P(x_1, x_2,x_3,  \cdots \,) = \left( \frac{x_1+x_2}{2},\frac{x_1+x_2}{2},
x_3, \cdots \, \right).\]
we show that $\Lambda_P = \{1,-1\}.$  Given $\la \in \bT$ such that
$T=P+ \la (I-P) $ is a surjective isometry, then
\[T (x_1, x_2,x_3,
\cdots \,)= \left( \frac{(\la +1)x_1+ (1-\la)x_2}{2},
\frac{(\la +1)x_2+ (1-\la)x_1}{2}, x_3, \cdots \, \right).\]
We recall that a surjective isometry on $\ell_{\infty}$, $S:\ell_{\infty}
\rightarrow \ell_{\infty}$ is of the form
\[ S(x_1, x_2,x_3,  \cdots \,)= (\mu_1 x_{\tau(1)}, \mu_2 x_{\tau(2)} , \cdots ),\]
with $\tau$ a bijection of $\mathbb{N}$  and $\{\mu_i\}$ is a sequence of
modulus 1 complex numbers.

Therefore  $T$ is an isometry if and only if $\la = \pm 1$.

\item  Let $P$ and $T$ on $(\mathbb{C}^3, \|\cdot \|_*)$ be defined as in example (\ref{gbp_ne_(I+R):2}).
Then
$\Lambda_P = \{1, exp(\frac{2 \pi i}{3}), exp (\frac{4 \pi i }{3})\}$.
Since, $T=P+exp(\frac{2 \pi i}{3})(I-P)$ is an isometry on $(\mathbb{C}^3, \|\cdot \|_*)$,
then $T^2=P+exp(\frac{4 \pi i}{3})(I-P)$ is also an isometry and $\Lambda_P \supseteq \{1,
exp(\frac{2 \pi i}{3}), exp (\frac{4 \pi i }{3})\}$. We now show that $\Lambda_P = \{1, exp(\frac{2 \pi i}{3}),
exp (\frac{4 \pi i }{3})\}$. As in example (\ref{gbp_ne_(I+R):2}), let $\la = exp(\frac{2 \pi i}{3})$.
Given $\lambda_0=a_0+i b_0$ of modulus 1, such that $\lambda_0 \notin \{1, exp(\frac{2 \pi i}{3}),
exp (\frac{4 \pi i }{3})\}$,  we set $S=P+\lambda_0 (I-P).$ Therefore,
\[ S(x,y,z)= \f{1}{3} (cx + d(y+z),cy + d(x+z),cz + d(x+y)),\] with
$c = 1+2 \la_0$ and $d = 1- \la_0$ and
$$ \| S(0,0,1)\|_* = \max \{ \| S(0,0,1)\|_{\infty}, \| TS(0,0,1)\|_{\infty}, \| T^2S(0,0,1)\|_{\infty} \}.$$
Now, $S(0,0,1) = \f{1}{3}(d,d,c)$, $TS(0,0,1)= \f{1}{3}(1- \la_0
\la,1- \la_0 \la,1+ 2 \la_0 \la)$ and $T^2S(0,0,1)= \f{1}{3}(1-
\la_0 \la^2,1- \la_0 \la^2,1+ 2 \la_0 \la^2)$. It is easy to see
that each of $|\f{1- \la_0}{3}|$, $|\f{1- \la_0 \la}{3}|$ and
$|\f{1- \la_0 \la^2}{3}|$ is strictly less than 1. Moreover, if any
of $| \f{1+2 \la_0}{3}|$, $|\f{1+ 2 \la_0 \la}{3}|$ or $|\f{1+ 2
\la_0 \la^2}{3}|$ is equal to 1, then $\la_0 = 1, \,
\la_0=\overline{\la}$ or $\la_0=\overline{\la}^2,$ respectively.
This leads to a contradiction. It also follows from calculations
already done for the  example (\ref{gbp_ne_(I+R):2}) that
$\|(0,0,1)\|_*=1.$ Therefore, $\|S(0,0,1)\|_* \neq \|(0,0,1)\|_*$
and hence $\la_0 \notin \La_P$.

\end{enumerate}

\end{ex}

The next corollary follows from our proof presented for Proposition \ref{mt1}.

\begin{cor}

Let $X$ be a Banach space. If the order of the multiplicative group
of a generalized bi-circular projection $P$ on $X$ is even then $P$
is the average of the identity with an isometric reflection.

\end{cor}

We also recall the definition of a generalized $n$-circular projection, cf. \cite{B}.
A projection $P$ on $X$ is generalized $n$-circular if and only if
there exists a surjective isometry $T$ such that $T^n=I$ and
\[ P = \f{I+T+T^2+ \cdots \, +T^{n-1}}{n}. \]

Another notion of generalized
$n$-circular projection was defined in \cite{AD} and it was shown there that
both the definitions are equivalent in $C(\Om)$,
where $\Om$ is a compact Hausdorff connected space.
In fact, they are equivalent in any space in which the GBPs are given as
the average of identity with an isometric reflection, see \cite{AD1}.

We observe that for a surjective linear map $T$ on $X$ such that
$T^n=I$ (not necessarily an isometry), $\f{I+T+T^2+ \cdots \,
+T^{n-1}}{n}$ is a projection. The same question applies to this
situation; which spaces support only $n$-circular projections
associated with surjective isometries?

We now show a result concerning existence of projections written
as a linear combination of operators with a cyclic property.

\begin{Def}

An operator $T$ on $X$ is  of order $k$
(a positive integer) if and only if $T^k=I$ and $T^i \neq I$ for any $i<k$.

\end{Def}

We observe that if $T$ is of order k, then $P= \frac{I+T+T^2+ \cdots +
T^{k-1}}{n}$ is a projection. The following proposition answers the
reverse question whether a combination of such a collection of operators
yields any projection.

Before stating our result we set some useful notation as introduced in the book by Michael Frazier,
\cite{fr}. We define $\rho= e^{-2 \pi i/k}.$  Then $\rho^{mn} = e^{-2 \pi mn i /k}$  and
$\rho^{-mn} = e^{2 \pi mn i /k}.$ In this notation, given a $k$-tuple $z= (z(0), \, \cdots, \, z(k-1))$
we set $\hat{z} (m) = \sum_{n=0}^{k-1} z(n) \rho^{mn}.$
We now denote by $W_k$ the $k$-square matrix with the $(i,j)$ entry equal to $\rho^{(i-1)(j-1)}$. In expanded form
\[ W_k= \left[ \begin{array}{cccccc} 1& 1& 1 &1 & \cdots & 1 \\ 1 & \rho & \rho^2 & \rho^3 & \cdots & \rho^{k-1} \\
1 & \rho^2 & \rho^4 & \rho^6 & \cdots & \rho^{2(k-1)} \\
1 & \rho^3 & \rho^6 & \rho^9 & \cdots & \rho^{3(k-1)} \\
\vdots & \vdots & & & & \vdots \\
1 & \rho^{k-1} & \rho^{2(k-1)} & \rho^{3(k-1)} & \cdots &
\rho^{(k-1)(k-1)} \end{array} \right]. \] Regarding  $z$ and
$\hat{z}$ as column vectors  we have  $\hat{z}= W_k z$. It is easy
to see that $W_k$ is invertible. The $(i,j)$-entry of  $W_k^{-1}$ is
equal to $\frac{1}{k}\overline{\rho}^{(i-1)(j-1)}.$ Frazier
designates  $\hat{z}$  the ``discrete Fourier transform" of $z$,
i.e., $\hat{z}=DFT (z)$, and $z$ is the ``inverse discrete Fourier
transform" of $\hat{z}$ , i.e., $z=W_k^{-1}\hat{z}=IDFT (\hat{z})$.
If $S$ is a subset of $\{0, \cdots , k-1\}$, we  denote by
$\delta_S$ the vector with components given by $\delta(i)=1 $  for
$i \in S$ and  $\delta(i)=0 $ otherwise.

\begin{prop} \label{mp1}

Let $X$ be a Banach space and $P$ a bounded operator on $X$. Let
$\lambda_0, \,\,\cdots, \, \lambda_{k-1}$ be nonzero complex numbers
and $P= \sum_{i=0}^{k-1} \lambda_i \, T^i,$ where $T$ is an operator of
order $k$. Then, $P$ is a projection if and only if $\lambda = (\lambda_0, \, \lambda_1, \, \cdots , \lambda_{k-1})$
is the IDFT of $\delta_S$, for some $S \subseteq \{0, \cdots , k-1\}$.

\end{prop}

\begin{proof}

If $P=\sum_{i=0}^{k-1} \lambda_i \, T^i$  and $T$ is an algebraic operator with annihilating polynomial $x^k-1$, a
Theorem due to Taylor (cf. \cite{ta} p. 317-318) asserts that
\[ T= Q_0 + \rho Q_1 + \cdots  + \rho^{k-1} Q_{k-1}\]
with $\{Q_0, \cdots , Q_{k-1}\}$ pairwise orthogonal projections. Since $T^i = Q_0 + \rho^i Q_1 + \cdots  + \rho^{i(k-1)} Q_{k-1}$
we conclude that $P= \alpha_0 Q_0+ \alpha_1 Q_1+\cdots +\alpha_{k-1} Q_{k-1}$ with the vector of scalars
$(\alpha_0, \cdots , \alpha_{k-1})$ equal to the $DFT (\lambda_0, \lambda_1, \cdots , \lambda_{k-1})$.
Since $P$ is a projection, i.e., $P^2=P$ and  $\{Q_0, \cdots , Q_{k-1}\}$ are pairwise orthogonal
projections we have that $\alpha_i^2=\alpha_i,$ for $i=0, \cdots,\, k-1.$ On the other hand,
\[ P= \sum_{i=0}^{k-1} \lambda_i \, T^i= \sum_{i=0}^{k-1} \,\left(\sum_{j=0}^{k-1} \lambda_j \rho^{ij} \right) Q_i,\]
thus for $i=0, \cdots , k-1$,  $\alpha_i=\sum_{j=0}^{k-1} \lambda_j \rho^{ij}$.
This implies that $(\lambda_0, \lambda_1,\, \cdots , \,\lambda_{k-1})$ is the $IDFT (\delta_S)$ for
some $S$ a subset of $\{0, \cdots , k-1\}.$

Conversely, we associate with $T$ the collection $Q_0, \cdots,
Q_{k-1}$ of $k$ pairwise orthogonal projections, such that  the
range of each $Q_i$ is the eigenspace associated with the eigenvalue
$ \rho^{i}.$  Then $ \delta_S(0) Q_0+\delta_S(1) Q_1 + \cdots +
\delta_S(k-1)Q_{k-1} =\sum_{i=0}^{k-1} \lambda_i \, T^i,$  and $P=
\delta_S(0) Q_0+\delta_S(1) Q_1 + \cdots + \delta_S(k-1)Q_{k-1}$ is
clearly a projection. This completes the proof.

\end{proof}

\section{Spaces of Vector-Valued Functions}

In this section we characterize generalized bi-circular projections on spaces of continuous
functions defined on a locally compact Hausdorff space. This characterization extends results
presented in \cite{B} and \cite{BJ1} for compact and connected Hausdorff spaces.
We recall a folklore lemma which  is very easy to prove.

\begin{lem} \label{lem1}

Let $X$ be a Banach space and $\la \in \bT \setminus \{1\}$. Then the
following are equivalent. \bla

\item $T$ is a bounded operator on $X$ satisfying $T^2 - (\la +
1)T + \la I = 0$.

\item There exists a projection $P$ on $X$ such that $P+\la
(I-P)=T$. \el

\end{lem}

\begin{thm} \label{thm2}

Let $\Om$ be a locally compact Hausdorff space, not necessarily connected,
and $X$ be a Banach space which has trivial centralizer. Let $P$ be a GBP
on $C_0(\Om,X)$. Then one and only one of the following holds. \bla

\item $P = \f{I+T}{2}$, where some $T $ is an isometric reflection on $C_0(\Om,X)$.

\item $Pf(\om) = P_{\om}(f(\om))$, where $P_{\om}$ is a generalized bi-circular projection on $X$.
\el

\end{thm}

\begin{proof}

Let $P + \la (I - P) = T$, where $\la \in \bT \setminus \{ 1 \} $ and
$T$ is an isometry on $C_0(\Om,X)$. From \cite{FJ1}, we see that
$T$ has the form $Tf(\om)= u_{\om}(f(\ph(\om)))$
for $\om \in \Om$ and $f \in C_0(\Om,X)$, where $u : \Om \ra
\mathcal{G}(X)$ continuous in strong operator topology and $\ph$
is a homeomorphism of $\Om$ onto itself. Here, $\mathcal{G}(X)$ is the
group of all surjective isometries on $X$. From Lemma \ref{lem1},
we have $T^2 - (\la + 1)T + \la I = 0$. That is
\begin{equation} \label{eq4} u_{\om} \circ u_{\ph(\om)}(f(\ph^2(\om)))
+(\la +1)u_{\om}(f(\ph(\om)))+\la f(\om)=0. \end{equation}
Let $\om \in \Om$. If $\ph(\om) \neq \om$,
then $\ph^2(\om)=\om$. For otherwise, there exists $h \in C_0(\Om)$ such
that $h(\om)=1$, $h(\ph(\om))=h(\ph^2(\om))=0$. For $f = h
\bigotimes x$, where $x$ is a fixed vector in $X$, Equation (\ref{eq4})
reduces to $\la = 0$, contradicting the assumption on $\la$. Now,
choosing $h \in C_0(\Om)$ such that $h(\om)=0$, $h(\ph(\om))=1$ we get
$\la = -1$. This implies that $u_{\om} \circ u_{\ph(\om)} = I$.
If $\ph(\om) = \om$ and $\ph$ is not the identity, then since we will
have the above case (i.e., $\ph(\om) \neq \om$) for some $\om's$, we
conclude that $\la = -1$. This again implies that $u^2_{\om}= I$.
Hence in both  cases $P$ will be of the form $\f{I+T}{2}$ and $T^2=I$.

If $\ph(\om) = \om$ for all $\om \in \Om$, then we will have from
Equation (2) \[u^2_{\om} - (\la +1)u_{\om} + \la I = 0.\] Thus
from Lemma \ref{lem1},  there exists a projection $P_{\om}$ on $X$
such that $P_{\om} + \la (I-P_{\om}) = u_{\om}$. Since $u_{\om}$
is an isometry, $P_{\om}$ is a GBP. Therefore, we have $Pf(\om) =
P_{\om}(f(\om)$. This completes the proof.

\end{proof}

\begin{cor} \label{ncor}

Let $\Om$ be a locally compact Hausdorff space (not necessarily connected)
and $P$ be a GBP on $C_0(\Om)$. Then one and only one of the following holds.
\bla

\item $P = \f{I+T}{2}$, where $T$ is an isometric reflection on $C_0(\Om)$.

\item P is a bi-circular projection. \el

\end{cor}

\begin{rem}

Similar results were proved in \cite{BJ1} for $C(\Om,X)$,
with $\Om$ connected. Here we extend those results to more general settings.

\end{rem}

It was proved in \cite{FJ} that if $(X_n)$ is a sequence of Banach
spaces such that every $X_n$ has trivial $L_\infty$ structure,
then any surjective isometry of $\bigoplus_{c_0}X_n$ is of the
form $(Tx)_n = U_{n\pi(n)}x_{\pi(n)}$ for each $x=(x_n) \in
\bigoplus_{c_0}X_n $. Here $\pi$ is a permutation of $\N$ and
$U_{n\pi(n)}$ is a sequence of isometric operators which maps
$X_{\pi(n)}$ onto $X_n$.

Suppose $P$ is a GBP on $\bigoplus_{c_0}X_n$, then similar techniques employed in the proof of
Theorem \ref{thm2} also prove the following result.

\begin{thm} \label{thm3}

Let $P$ is a a generalized bi-circular projection on $\bigoplus_{c_0}X_n$. Then one and only one of
the following holds. \bla

\item $P = \f{I+T}{2}$, where $T$ is an isometric reflection on $\bigoplus_{c_0}X_n$.

\item $(Px)_n = P_n x_n$ where $P_n$ is a generalized bi-circular projection on $X_n$.

\el

\end{thm}

\begin{ack}

The authors are very grateful to an anonymous referee for pointing
out an error on our original proof for Proposition \ref{mp1}. This
comment lead to a more general result. The work on this paper was
started during the first author's visit to the University of Memphis
for attending the International Conference on Mathematics and
Statistics in May 2012. He wishes to thank Prof.Fernanda Botelho and
Prof.James Jamison for their kind hospitality and for providing the
excellent research environment that stimulated this collaboration.

\end{ack}

\end{document}